\documentclass[a4paper,10pt]{amsart}
\usepackage[utf8]{inputenc}
\usepackage{amsmath, amssymb, amsthm, graphicx, wrapfig}

\newtheorem{Prop}{Proposition}[section]
\newtheorem{Lemma}[Prop]{Lemma}

\newtheorem{Thm}{Theorem}
\newtheorem{Def}[Prop]{Definition}

\newtheorem{Rem}[Prop]{Remark}

\newcommand{\WL}{\operatorname{WL}}
\newcommand{\WR}{\operatorname{WR}}
\newcommand{\Wie}{\operatorname{W}}
\numberwithin{equation}{section}

\title{Wieland gyration for triangular fully packed loop configurations}
\author{Sabine Beil}
\address{Sabine Beil, Universit\"at Wien, Fakult\"at f\"ur Mathematik, Oskar-Morgenstern-Platz~1, 1090 Wien, Austria}
\email{sabine.beil@univie.ac.at}
\author{Ilse Fischer}
\address{Ilse Fischer, Universit\"at Wien, Fakult\"at f\"ur Mathematik, Oskar-Morgenstern-Platz~1, 1090 Wien, Austria}
\email{ilse.fischer@univie.ac.at}
\author{Philippe Nadeau}
\address{Philippe Nadeau, CNRS, Institut Camille Jordan, Universit\'e Lyon 1, 43 boulevard du 11 novembre 1918, 69622 Villeurbanne cedex, France }
\email{nadeau@math.univ-lyon1.fr}
\thanks{Supported by the Austrian Science Foundation FWF, START grant Y463. PN was also  supported by the French Research Agency ANR, project  ANR-11-JS02-001.
}

\begin{document}

\begin{abstract}
Triangular fully packed loop configurations (TFPLs) emerged as auxiliary objects in the study of fully packed loop configurations on a square (FPLs) corresponding to link patterns with a large number of nested arches. 
Wieland gyration, on the other hand, was invented to show the rotational invariance of the numbers $A_\pi$ of FPLs corresponding to a given link pattern $\pi$. The focus of this article is the definition and study of 
Wieland gyration on \mbox{TFPLs}. We show that the repeated application of this gyration eventually leads to a configuration that is left invariant. We also provide a characterization of such stable configurations. 
Finally we apply our gyration to the study of TFPL configurations, in particular giving new and simple proofs of several results. 
\end{abstract}

\maketitle

\section*{Introduction}

Triangular fully packed loop configurations (TFPLs) first appeared in the study of ordinary fully packed loop configurations (FPLs). There they were used to show that the number of FPLs
corresponding to a given link pattern with $m$ nested arches is a polynomial in $m$, see~\cite{CKLN}. It soon turned out that TFPLs possess a number of nice properties, which made them worthy objects of study by themselves. For instance, they can be seen as a generalized model for Littlewood--Richardson coefficients, thereby establishing an unexpected link to algebra. This was first proven in~\cite{Nadeau2} by a convoluted argument and later in~\cite{TFPL} in a direct combinatorial manner and in a more general setting. Other combinatorial aspects of TFPLs, many of them still conjectural, are studied in~\cite{Nadeau1,Thapper}.
 
In 2000 Wieland~\cite{Wieland1} invented the operation on FPLs which bears his name. The \emph{Wieland gyration} was used to prove the rotational invariance of the numbers $A_\pi$ of FPLs corresponding to a given link 
pattern $\pi$. It was later heavily used by Cantini and Sportiello~\cite{CanSport} to prove the Razumov--Stroganov conjecture. It also  came up in connection with TFPLs already in~\cite{Nadeau1} following work of~\cite{Thapper}. \medskip

The main contribution of this article is the explicit definition of Wieland gyration for TFPLs  together with a detailed study of some of its properties. 

While the usual Wieland gyration of FPLs is an involution, our \emph{left-Wieland gyration} $\WL$ acting on TFPLs is not. By a finiteness argument, the sequence \hbox{$(\WL^m(f))_{m\geq 0}$} is
eventually periodic. In Theorem \ref{Thm:EventuallyStable}, it will be proven that the length of the period is always one, which means one always reaches a TFPL which is invariant under left-Wieland gyration. In fact, if $N$ is the size of $f$, then less that $2N$ iterations of $\WL$ will suffice to obtain such \emph{stable} configurations. A key step in the proof of Theorem \ref{Thm:EventuallyStable} is to classify these stable TFPLs. It turns out that this depends solely on the occurrence of a certain type of edges called \emph{drifters}: this is the content of Theorem \ref{Thm:StableTFPL}. These results also hold for \emph{right-Wieland gyration}.

Now to each TFPL are assigned three binary words $u$, $v$ and $w$  that encode its boundary conditions. Such binary words $\sigma$ are naturally associated with Young diagrams $\lambda(\sigma)$, and by the results of \cite{Nadeau2,TFPL}, TFPLs with boundary $(u,v;w)$ such that \hbox{$|\lambda(u)|+|\lambda(v)|=|\lambda(w)|$} are enumerated by the Littlewood-Richardson-coefficient $c_{u,v}^w$. We will show that such TFPLs are stable. In general, the boundary $(u,v;w)$ of a TFPL has to satisfy $|\lambda(u)|+|\lambda(v)|\leq |\lambda(w)|$: this was proven in~\cite{Thapper} using Wieland gyration and a certain degree argument, and later reproven in a combinatorial fashion in~\cite{TFPL}. Here we will use left- and right-Wieland gyrations to give a simple proof of this inequality. \medskip

The paper is divided as follows. In Section~\ref{Section:definitions_results} we recall the definition of FPLs and TFPLs as well as elementary properties of binary words and Young diagrams. 
Section~\ref{Section:Wieland} contains the definition of our main construction, the left-Wieland gyration acting on TFPLs, based on Wieland's original definition. It is introduced in Definition~\ref{def:LeftWieland} and 
we give its first properties, culminating in Theorem~\ref{Thm:WielandBijectiveTFPL}. We can then state the theorems about stability of TFPLs, namely Theorems~\ref{Thm:EventuallyStable} and~\ref{Thm:StableTFPL}, which 
are proven in Section~\ref{Section:stable}. Finally, Section~\ref{Section:applications} contains applications of our gyration to enumerative questions concerning TFPLs.

%%%%%%%%%%%%%%%%
\section{Definitions and elementary properties}
\label{Section:definitions_results}
%%%%%%%%%%%%%%%%%

In this section we recall the definitions of FPLs and TFPLs, and the binary words attached to the {\em boundary} of a TFPL with the necessary conditions they must satisfy.

\subsection{Fully packed loop configurations}\label{Section:FPL}
Fully packed loop configurations first came up in statistical physics; they are an alternative representation of \emph{six-vertex model} configurations which are in one-to-one correspondence with \emph{square-ice} configurations, see for example~\cite{FPL1} and~\cite{Wieland1}. Furthermore, they are in bijection 
with alternating sign matrices and other combinatorial configurations, cf.~\cite{Propp}.

 We start with the graph $G_n$, which is defined as the square grid with $n^2$ vertices together with $4n$ \emph{external edges}. 
The $(n+1)^2$ unit squares of this grid, including external cells that have two or three surrounding
edges only, are said to be the \emph{cells} of $G_n$. They are partitioned into odd and even cells in a chessboard manner where by convention the 
cells on the Northwest-Southeast diagonal are odd. In Figure~\ref{G_8} the graph $G_8$ together with its odd and even cells is depicted. 

\begin{figure}[tbh]
 \centering
 \includegraphics[width=.4\textwidth]{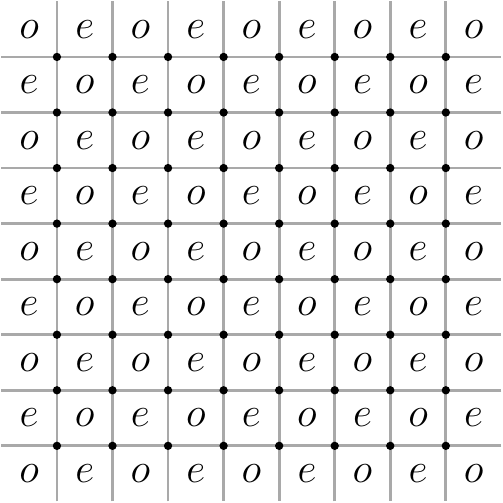}
 \caption{The graph $G_8$ with its odd and even cells.}
 \label{G_8}
\end{figure}

\begin{Def} A \emph{fully packed loop configuration} (FPL) of size $n$ is a subgraph $F$ of $G_n$ satisfying that
\begin{enumerate}
 \item each vertex of $G_n$ is incident to two edges of $F$, and
 \item precisely every other external edge belongs to $F$. 
\end{enumerate} 
\end{Def}

Given an FPL $F$ a \emph{cell of $F$} is defined as a cell of $G_n$ together with those of its surrounding edges that belong to $F$. An example of an FPL is given in Figure~\ref{ExampleFPL}. 
In a natural way, every FPL defines a non-crossing matching of the occupied external edges -- its so-called \emph{link pattern} -- by matching those which are joined by a path. 
\begin{figure}[tbh]
\centering
\includegraphics[width=.4\textwidth]{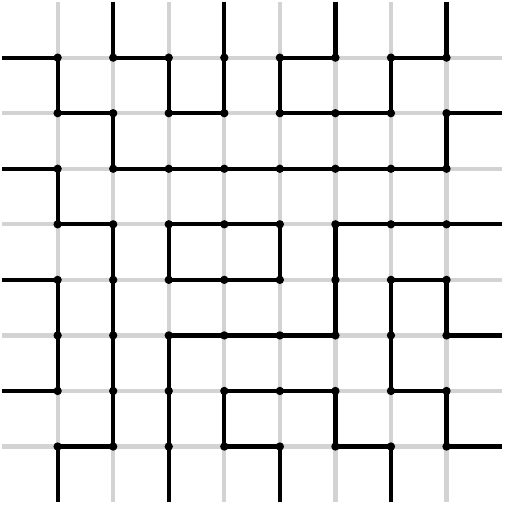}
\caption{A FPL of size 8.}
\label{ExampleFPL}
\end{figure}

In the course of the study of FPLs corresponding to fixed link patterns with a sufficiently large number of \emph{nested arches}, TFPLs first occurred:
such FPLs admit a combinatorial decomposition, in which TFPLs naturally arise. This combinatorial decomposition first came up in the course of the proof in~\cite{CKLN} of a conjecture in~\cite{Zuber} stating that if we introduce $m$ nested arches in a fixed link pattern $\pi$ then the number of FPLs corresponding to this link pattern is a polynomial 
function in $m$ as $m$ varies.

 %%%%%%%%%%%%%%%%%%%%%%%%%%%
\subsection{Triangular fully packed loop configurations}
\label{subsection:TFPL}

To give the definition of triangular fully packed loop configurations, we need the following graph:
\begin{Def}[The graph $G^N$] Let $N$ be a positive integer. The graph $G^N$ is defined as the induced subgraph of the square grid made up of $N$ consecutive centered rows of 
\hbox{$3,5,\dots, 2N+1$} vertices from top to bottom together with $2N+1$ vertical \emph{external} edges incident to the $2N+1$ bottom vertices. 
\end{Def}
\begin{figure}[tbh]
\centering
\includegraphics[width=.7\textwidth]{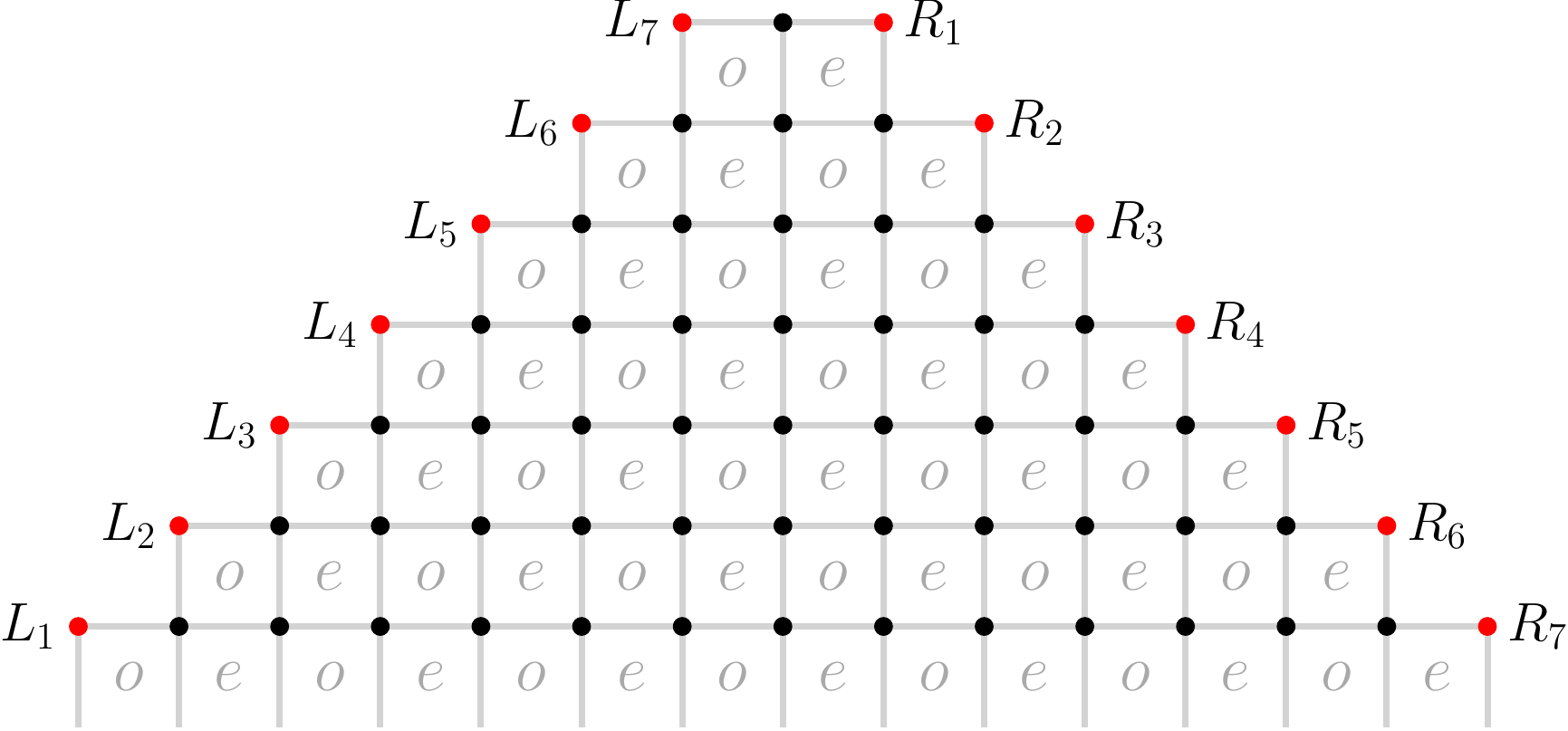}
\caption{The graph $G^7$ with its odd and even cells.}
\label{Fig003}
\end{figure}

In the following, let \hbox{$\mathcal{L}^N=\{L_1,L_2,\dots,L_N\}$} (resp. \hbox{$\mathcal{R}^N=\{R_1,R_2,\dots,R_N\}$}) be the set made up of the leftmost (resp. rightmost) vertices of the $N$ rows of $G^N$, where the vertices are numbered from left to right.
Furthermore, the $N(N+1)$ unit squares of $G^N$, including external unit squares that have three surrounding
edges only, are said to be the \emph{cells} of $G^N$. They are partitioned into \emph{odd} and \emph{even} cells in a chessboard manner where by convention the top left 
cell of $G^N$ is odd. In Figure~\ref{Fig003} the graph $G^7$ together with its odd and even cells is pictured.

\begin{Def}[\cite{TFPL}] \label{defi:tfpl}
 Let $N$ be a positive integer.
A \emph{triangular fully packed loop configuration} (TFPL) of size $N$ is a subgraph $f$ of $G^N$ such that: 
\begin{enumerate}
 \item[(i)] Every other external edge starting with the second one belongs to $f$. 
 \item[(ii)] The $2N$ vertices in $\mathcal{L}^N\cup\mathcal{R}^N$ have degree 0 or 1.
 \item[(iii)] All other vertices of $G^N$ have degree 2.
 \item[(iv)] A path in $f$ neither connects two vertices of $\mathcal{L}^N$ nor two vertices of $\mathcal{R}^N$.
\end{enumerate}
\end{Def}
\begin{figure}[tbh]
\centering
\includegraphics[width=.7\textwidth]{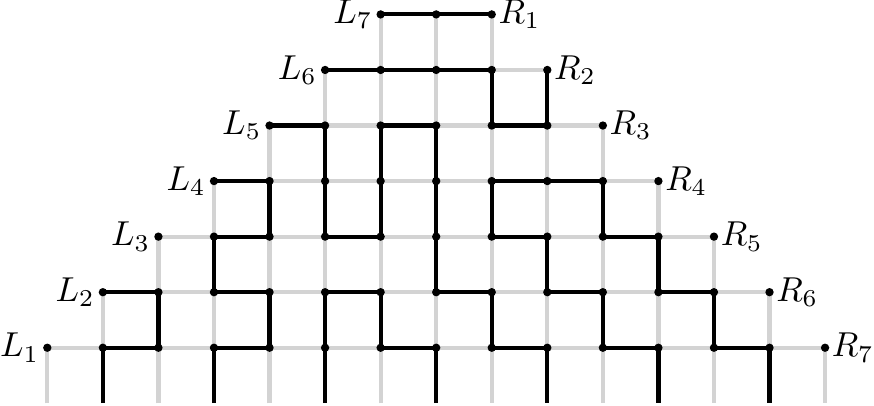}
\caption{A TFPL of size 7.}
\label{Fig002}
\end{figure}
An example of a TFPL is given in Figure~\ref{Fig002}. Similar to FPLs, a \emph{cell of $f$ }is a cell of $G^N$ together with those of its surrounding edges that belong to $f$. A cell is called \emph{interior} if it 
does not contain a vertex in $\mathcal{L}^N \cup \mathcal{R}^N$.\\

By {\em binary words} we refer to words \mbox{$\sigma=\sigma_1\cdots\sigma_N$} where $\sigma_i\in\{0,1\}$ for each $1\leq i\leq N$.
In the following, the number of occurrences of $1$ (resp. $0$) in a binary word $\sigma$ is denoted by $\vert\sigma\vert_1$ (resp. $\vert\sigma\vert_0$).
To each TFPL of size $N$ a triple of binary words of length $N$ is assigned as follows:

\begin{Def}
Let $f$ be a TFPL of size $N$. The \emph{boundary} of $f$ is a triple $(u,v;w)$ of binary words of length $N$ defined as follows: 
\begin{enumerate}
 \item $u_i=1$ if and only if $L_i\in\mathcal{L}^N$ has degree 1,
 \item $v_i=1$ if and only if $R_i\in\mathcal{R}^N$ has degree 0 and
 \item $w_i=1$ if and only if  the $i$-th external edge in $f$ -- counted from left to right -- is connected by a path in $f$ either with a vertex in $\mathcal{L}^N$ or with an external edge to 
 its left.   
\end{enumerate}
\end{Def}
The set of all TFPLs with boundary $(u,v;w)$ is denoted by $T_{u,v}^w$ and its cardinality by $t_{u,v}^w$.
For example, the triple \hbox{$(0101111,0011111;1101101)$} is the boundary of the TFPL depicted in Figure~\ref{Fig002}.
A triple $(u,v;w)$ that is the boundary of a TFPL has to fulfill certain necessary conditions. To formulate them
the following standard result is needed:

\begin{Prop}[\cite{Nadeau2}]\label{Prop005}
For given non-negative integers $k$ and $\ell$ the following two sets are in bijection:
\begin{enumerate}
 \item[(i)] the set of binary words $\sigma$ satisfying $\vert\sigma\vert_0=k$ and $\vert\sigma\vert_1=\ell$ and
 \item[(ii)] the set of Young diagrams fitting in the rectangle with $k$ rows and $\ell$ columns.
\end{enumerate}
\end{Prop}

In Figure~\ref{Fig4}, an example for the bijection between binary words and Young diagrams is given. The Young diagram corresponding to a binary
word $\sigma$ is denoted by $\lambda(\sigma)$. Furthermore, $\lambda(\tau)\subseteq\lambda(\sigma)$ means that the Young diagram $\lambda(\tau)$ is included in the Young diagram 
$\lambda(\sigma)$, and $\vert \lambda(\sigma)\vert$ denotes the number of cells of the Young diagram $\lambda(\sigma)$. Note that $\vert \lambda(\sigma) \vert$ coincides with the number of inversions of the binary word $\sigma$.

\begin{Thm}[\cite{CKLN,Thapper,TFPL}]
In order for a TFPL configuration with boundary $(u,v;w)$ to exist, the following must be satisfied:
\begin{align}
|u|_0 = &|v|_0 = |w|_0,
\label{Necessary1}\\
\lambda(u)\subseteq\lambda(w) &\textnormal{ and }\lambda(v)\subseteq\lambda(w),
\label{Necessary2}\\
\vert \lambda(u)\vert + &\vert \lambda(v)\vert\leq\vert \lambda(w)\vert. 
\label{Necessary3}
\end{align} 
\end{Thm} 

Conditions~\eqref{Necessary1}and~\eqref{Necessary2} are reasonably easy to prove. In Section~\ref{Section:applications}, we will provide a new proof 
of Condition~\eqref{Necessary3} using Wieland gyration on TFPLs.

\begin{figure}[tbh]
\centering
\includegraphics[width=1\textwidth]{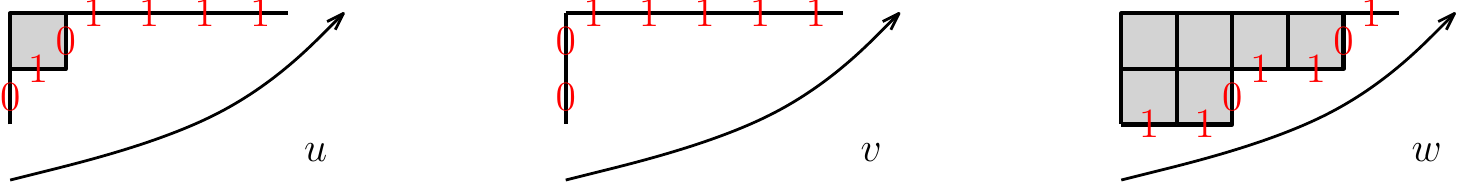}
\caption{The Young diagrams which correspond to the boundary $(0101111,0011111;1101101)$ of the TFPL in Figure~\ref{Fig002}.}
\label{Fig4}
\end{figure}

To end this section, we need certain skew shapes which play an important role in the context of left- and right-Wieland gyration. A skew shape is said to be a \emph{horizontal strip} ( \emph{resp. a vertical strip}) if each of its columns (\emph{resp.} rows) contains at most one cell.  Examples are given in Figure~\ref{Fig008}.

 Consider  two binary words $\sigma$ and $\tau$ satisfying $\vert \sigma\vert_1=\vert \tau\vert_1$ and $\vert\sigma\vert_0=\vert \tau\vert_0$. Then the skew shape \hbox{$\lambda(\tau)/\lambda(\sigma)$} is a horizontal strip (\emph{resp.} a vertical strip) if and only if for each \hbox{$j\in\{1,\dots,\vert \sigma\vert_1\}$} the following holds: 
{\it If $\sigma_i$ is the \hbox{$j$-th} one (\emph{resp.} zero) in $\sigma$, then $\tau_{i-1}$ or $\tau_i$ (\emph{resp.} $\tau_{i}$ or $\tau_{i+1}$) is the \hbox{$j$-th} one  (\emph{resp.} zero)  in $\tau$.}

\begin{figure}[tbh]
 \centering
  \includegraphics[width=.6\textwidth]{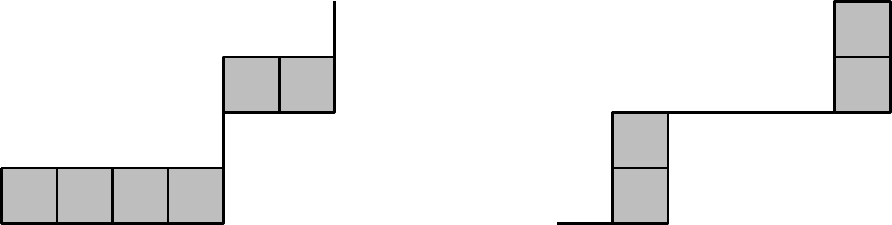}
  \caption{\label{Fig008} The horizontal strip $\lambda(1111001100) / \lambda(0111100110)$ and the vertical strip $\lambda(1100111100)/\lambda(1001111001)$.}
\end{figure}

 In the following, if the skew shaped Young diagram \hbox{$\lambda(\tau)/\lambda(\sigma)$} is a horizontal strip (\emph{resp.} a vertical strip), we will write \hbox{$\sigma\stackrel{\mathrm{h}}{\longrightarrow}\tau$} (\emph{resp.} \hbox{$\sigma\stackrel{\mathrm{v}}{\longrightarrow}\tau$}).

 %%%%%%%%%%%%%%%%%%%%%%%%%%%
\section{Wieland gyration for TFPLs}
\label{Section:Wieland}
%%%%%%%%%%%%%%%%%%%%%%%%%%%

In this subsection the definitions of left- and right-Wieland gyration for TFPLs are given and some first
properties are derived. The starting point is the definition of Wieland gyration for FPLs. It is composed of local operations on all \emph{active} cells of an FPL: the active cells of an FPL can be chosen to be either all its odd cells or all its even cells. 

Now, let $F$ be an FPL and $c$ be an active cell of $F$. Then we must distinguish two cases, namely whether $c$ contains precisely two edges of $F$ on opposite sides or not. If this is the case, Wieland gyration $\Wie$ leaves $c$ invariant. Otherwise, the effect of $\Wie$ on $c$ is that edges and non-edges of $F$ are exchanged.  In Figure~\ref{Wieland}, the action of $\Wie$ on an active cell is illustrated. The result of applying $\Wie$ to each active cell of $F$ is said to be the image of $F$ under Wieland gyration and is denoted by $\Wie(F)$.

 \begin{figure}[tbh]
  \includegraphics[width=.4\textwidth]{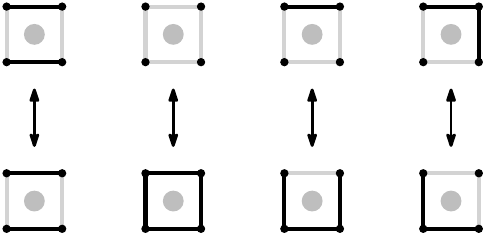}
  \caption{Up to rotation, the action of $\Wie$ on the active cells of an FPL.}
  \label{Wieland}
 \end{figure} 

In Figure~\ref{WielandFPLExample} the image of the FPL depicted in Figure~\ref{ExampleFPL} under Wieland gyration with the odd cells being active is pictured.\\
 \begin{figure}[tbh]
  \includegraphics[width=.9\textwidth]{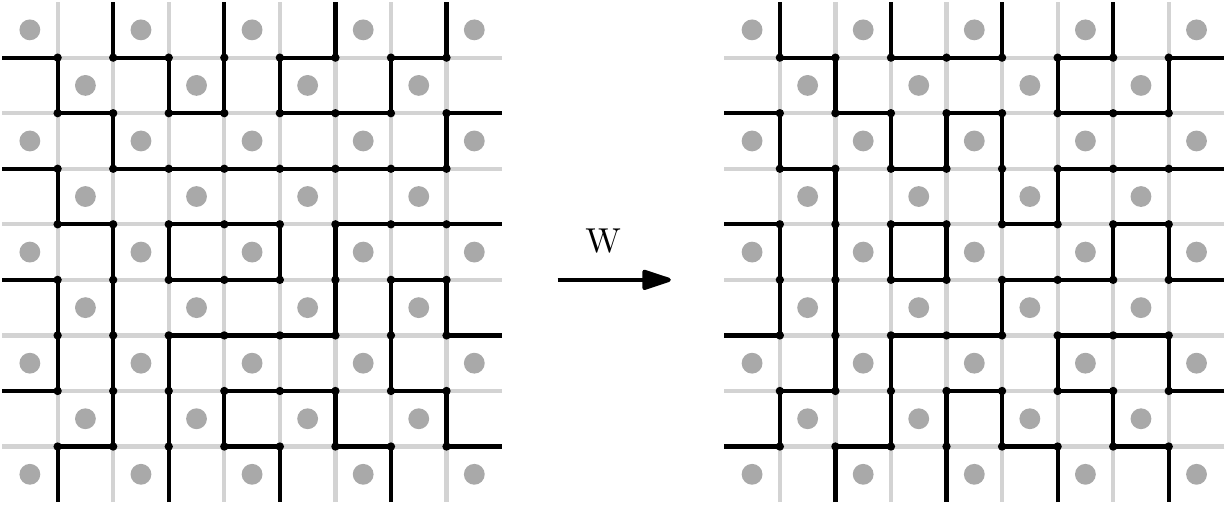}
  \caption{The image of the FPL depicted in Figure~\ref{ExampleFPL} under Wieland gyration with the odd cells being active.}
  \label{WielandFPLExample}
 \end{figure} 

Wieland gyration as it will be defined for TFPLs is based on the operation $\Wie$. As active cells of a TFPL can be chosen either all its odd cells or all its even 
cells. Choosing all odd cells as active cells will lead to what will be defined as left-Wieland gyration, whereas choosing all even cells as active cells will 
lead to what will be defined as right-Wieland gyration.

\begin{Def}[Left-Wieland gyration] 
\label{def:LeftWieland}
Let $f$ be a triangular fully packed loop configuration with left boundary word $u$, and let $u^-$ be a binary word such 
that $u^-\stackrel{\mathrm{h}}{\longrightarrow} u$. The \emph{image of $f$ under left-Wieland gyration with respect to $u^-$} is determined as follows:
\begin{enumerate}
 \item Insert a vertex $L'_i$ to the left of $L_i$ for $1\leq i\leq N$. Then run through the occurrences of ones in $u^-$: Let 
$\{i_1 < i_2 < \ldots < i_{N_1}\} = \{ i | u^-_i=1\}$.  
\begin{enumerate}
\item If $u_{i_j}$ is the $j$-th one in $u$, add a horizontal edge between $L'_{i_j}$ and $L_{i_j}$.
 \item If $u_{i_{j}-1}$ is the $j$-th one in $u$, add a vertical edge between $L'_{i_j}$ and $L_{{i_j}-1}$.
 \end{enumerate}
 \item Apply Wieland gyration to each odd cell of $f$ .
 \item Delete all vertices in $\mathcal{R}^N$ and their incident edges.
 \end{enumerate}
After shifting the whole construction one unit to the right, one obtains the desired image $\WL_{u^-}(f)$. 
\end{Def}

In the case $u^-=u$, we will simply write $\WL(f)$ and speak of the \emph{image of $f$ under left-Wieland gyration}.

\begin{figure}[tbh]
\begin{center}
\includegraphics[width=.7\textwidth]{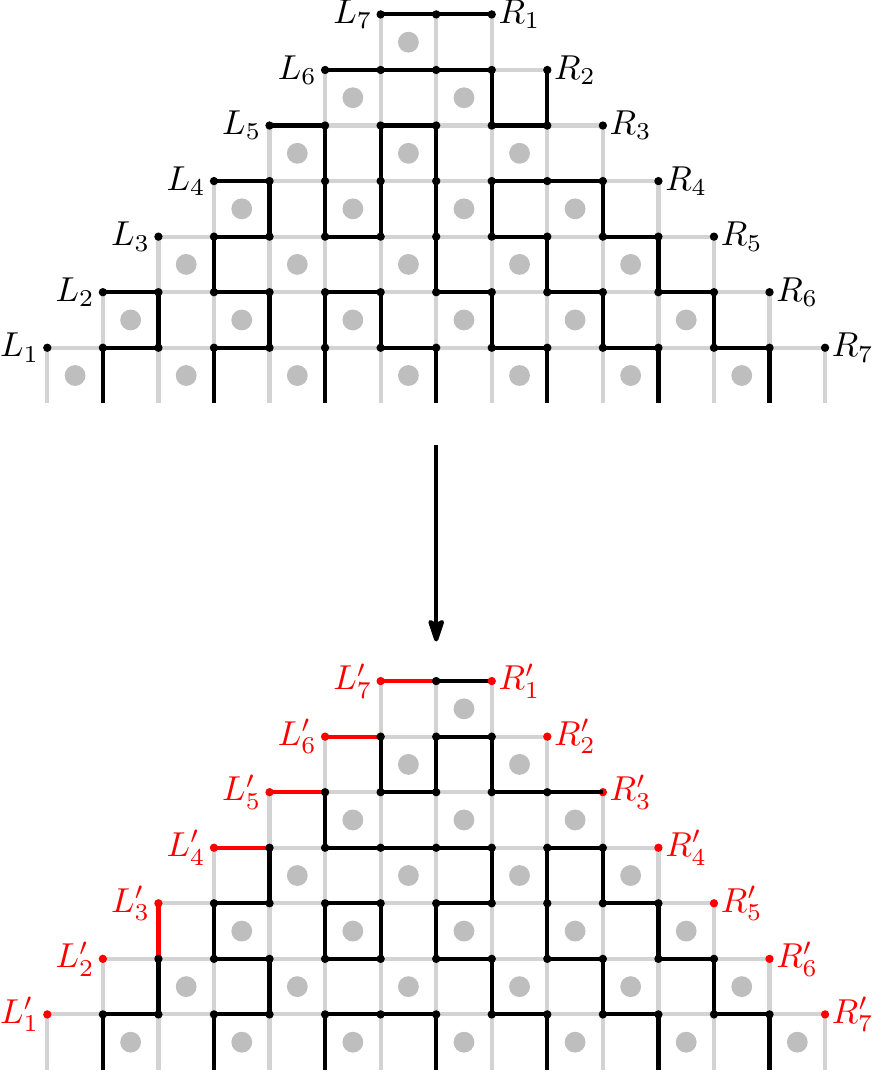}
\caption{The TFPL depicted in Figure~\ref{Fig002} and its image under left-Wieland gyration with respect to \hbox{$0011111$}.}
\label{WielandexampleTFPL}
\end{center}
\end{figure}

In the following, to distinguish between vertices in $f$ and in $\WL_{u^-}(f)$ the following notation is chosen: when regarding the image under left-Wieland gyration with respect to $u^-$, we will write $x'$ for each 
vertex $x$ of $G^N$ (before the shifting is performed).

In Figure~\ref{WielandexampleTFPL} the TFPL depicted in Figure~\ref{Fig002} with its odd cells marked by gray discs and its image under left-Wieland gyration with 
respect to \hbox{$0011111$} are pictured. It is a TFPL with boundary \hbox{$(0011111,0101111;1101101)$}. Note that the left boundary of the TFPL pictured in \hbox{Figure~\ref{Fig002}} is $0101111$ and \hbox{$0011111\stackrel{\mathrm{h}}{\longrightarrow}0101111$}. Also, the new right boundary $0101111$ and the right boundary $0011111$ of the preimage satisfy that \hbox{$0011111\stackrel{\mathrm{v}}{\longrightarrow}0101111$}. This turns out to hold in general:

\begin{Prop}\label{imageWielandTFPL}
Let $f$ be a TFPL with boundary $(u,v;w)$ and let $u^-$ be a binary word satisfying $u^-\stackrel{\mathrm{h}}{\longrightarrow}u$. 
Then $\WL_{u^-}(f)$ is a TFPL with boundary $(u^-,v^+;w)$ where $v^+$ is a binary word satisfying  $v\stackrel{\mathrm{v}}{\longrightarrow}v^+$.
\end{Prop}

\begin{proof} 
First, we have to check that $\WL_{u^-}(f)$ indeed is a TFPL, i.e. the four conditions in Definition~\ref{defi:tfpl} must be satisfied. By definition, the vertices $L'_1,L'_2,\dots,L'_N$ have degree $0$ or $1$. For the degree of $R_i'$ to be $2$
in $\WL_{u^-}(f)$, the vertex to the left of $R_i$ would need to be adjacent both to $R_{i-1}$ and $R_i$ in $f$, which is excluded since no path in $f$ joins two vertices 
in $\mathcal{R}^N$ by Definition~\ref{defi:tfpl}(iv).
Thus, the vertices $R_1',R_2',\dots,R_N'$ have degree $0$ or $1$ in $\WL_{u^-}(f)$. All other vertices have degree $2$ in $\WL_{u^-}(f)$ since they simply come from the application of $\Wie$ to cells of $f$. Finally 
let $f'$ denote the configuration that is obtained before the vertices of $\mathcal{R}^N$ are deleted.
Since Wieland gyration preserves the connectivity of path endpoints in each active cell, this is also true in $f'$. Thus, a path in $\WL_{u^-}(f)$ neither joins two vertices in $\mathcal{L}^{N\prime}$ nor two vertices 
in $\mathcal{R}^{N\prime}$ and Definition~\ref{defi:tfpl}(iv) is satisfied.

It remains to check the assertion on the boundary.  The left boundary of $\WL_{u^-}(f)$ is $u^-$ by construction. The right boundary $v^+$ of $\WL_{u^-}(f)$ satisfies  $v\stackrel{\mathrm{v}}{\longrightarrow}v^+$ by 
Proposition \ref{PropRightBoundary} below  and the characterization of pairs $\sigma,\sigma^+$ of binary words satisfying $\sigma\stackrel{\mathrm{v}}{\longrightarrow}\sigma^+$ at the end of Section~\ref{Section:definitions_results}. Finally, the bottom boundary of $\WL_{u^-}(f)$ is $w$ because Wieland gyration preserves the connectivity of path endpoints in each active cell. 
\end{proof}
 
The lemma below treats the effects of left-Wieland gyration along the right boundary of a TFPL. 

\begin{Lemma}\label{WielandRightBoundaryTFPL} Let $f,u^-,v^+$ be as in Proposition~\ref{imageWielandTFPL}. Then  $v^+\neq v$ if and only if there exists a vertex in $\mathcal{R}^N$ which is incident to a vertical edge of $f$.
\end{Lemma}

\begin{proof}
We denote by $x_s$ the vertex to the left of $R_s$ for all $1\leq s\leq N$, and write $f'$ to denote $\WL_{u^-}(f)$.

Let $f$ be a TFPL with a vertex $R_j$ incident to a vertical edge, and pick $j$ minimal. Then $x_j$ is necessarily adjacent both to the vertex to its left and to the vertex below, so by Wieland gyration  $R'_j$ is of 
degree $0$ in $f'$. Since $R_j$ is of degree $1$ this shows $v\neq v^+$.

Conversely, suppose that $v^+\neq v$. By Proposition~\ref{imageWielandTFPL} there
exists necessarily $j\in\{1,2,\dots,N-1\}$ such that $v_{j}=0$ and $v^+_j=1$.
$R_j'$ is of degree $0$ in $f'$, so $x_j$ is adjacent
 in $f$ both to the vertex to its left and to the vertex below it. Since $R_j$ is of degree $1$,  it is necessarily incident to a vertical edge.
\end{proof}
As a byproduct of the previous proof, one can in fact precisely describe the right boundary $v^+$ as follows:
\begin{Prop}\label{PropRightBoundary} 
 Conserve the hypotheses of Lemma~\ref{WielandRightBoundaryTFPL}. For each $i$ such that $R_i$ is adjacent to a horizontal edge (\emph{resp.} a vertical edge) then $v^+_i=0$ (\emph{resp.} $v^+_{i+1}=0$). All other values $v^+_j$'s are equal to $1$.
\end{Prop}

{\noindent \bf Right-Wieland gyration.} In the definition of left-Wieland gyration, the active cells are all odd cells of a TFPL. When selecting all even cells of a TFPL as active cells, \emph{right-Wieland gyration} 
is obtained. It depends on a binary word $v^-$ satisfying $v^-\stackrel{\mathrm{v}}{\longrightarrow}v$ that encodes what happens along the right boundary of a TFPL with right boundary $v$ and is denoted by $\WR_{v^-}$ 
repsectively $\WR$ if $v^-=v$. It is defined in an obvious way as the symmetric version of left gyration, and we shall simply illustrate it on an example in Figure~\ref{InverseWielandExampleTFPL}. 

\begin{figure}[tbh]
\includegraphics[width=.7\textwidth]{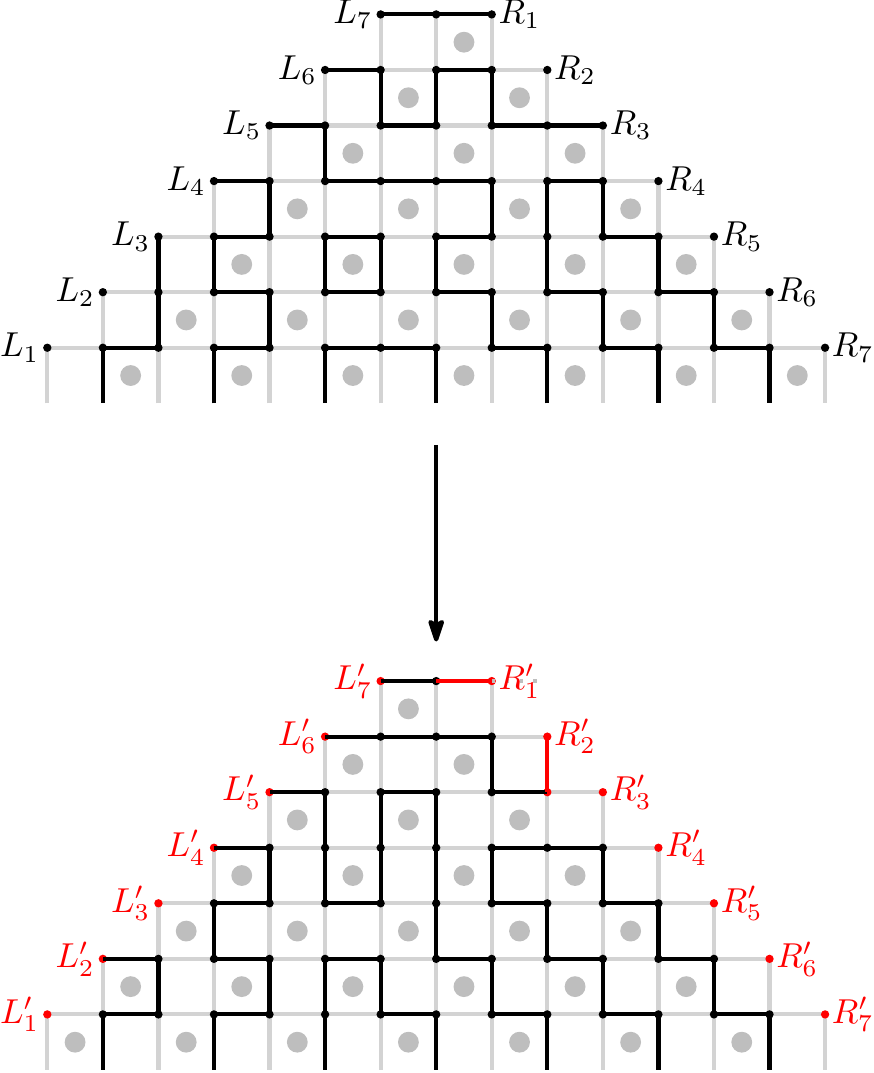}
\caption{A TFPL  and its image under right-Wieland gyration with respect to \hbox{$0011111$}.}
\label{InverseWielandExampleTFPL}
\end{figure}

There are immediate symmetrical versions of Propositions~\ref{imageWielandTFPL} and~\ref{PropRightBoundary} for $\WR$ which we record:

\begin{Prop}\label{ImageRightWieland}
The image of a TFPL with boundary $(u,v;w)$ under right-Wieland gyration with respect to $v^-$ is a TFPL with boundary $(u^+,v^-;w)$ where $u^+$ is a binary word satisfying $u\stackrel{\mathrm{h}}{\longrightarrow}u^+$. 
\end{Prop}

\begin{Prop}\label{PropLeftBoundary} Keep the notations of the previous proposition. 
For each index $i$ such that $L_i$ is adjacent to a horizontal edge (\emph{resp.} a vertical edge), there holds $u^+_i=1$ (\emph{resp.} $u^+_{i-1}=1$). All other values $u^+_j$'s are equal to $0$.
\end{Prop}

Given a TFPL with right boundary $v$, the effect of left-Wieland gyration along the right boundary of the TFPL is inverted by right-Wieland gyration with respect to $v$. On the other hand, given a TFPL with left 
boundary $u$ the effect of right-Wieland gyration along the left boundary is inverted by left-Wieland gyration with respect to $u$. Since Wieland gyration is an involution on each cell, it follows:

\begin{Thm}\label{Thm:WielandBijectiveTFPL} \begin{enumerate}
             \item Let $f$ be a TFPL with boundary $(u^+,v;w)$ and $u$ be a binary word such that $u\stackrel{\mathrm{h}}{\longrightarrow}u^+$. 
             Then 
             \[
             \WR_v(\WL_{u}(f))=f.
             \]
             \item Let $f$ be a TFPL with boundary $(u,v^+;w)$ and $v$ be a binary word such that 
             $v\stackrel{\mathrm{v}}{\longrightarrow}v^+$. Then 
             \begin{equation}
             \notag \WL_u(\WR_{v}(f))=f.
             \end{equation}
             
            \end{enumerate}

\end{Thm}

\begin{Rem}
It is perhaps useful to point out that $\WR(\WL(f))\neq f$ in general. Indeed by Lemma~\ref{WielandRightBoundaryTFPL} equality will hold precisely when all vertices $R_i$ of degree one are adjacent to horizontal edges.
\end{Rem}

In Section~\ref{Section:stable}, we will study the behaviour of TFPLs under iterated applications of $\WL$.
In Figure~\ref{Example:EventuallyStable}, an example of a TFPL to which left-Wieland gyration is repeatedly applied is depicted: one checks that the last TFPL in the sequence is  invariant by left-Wieland gyration. 
In the following, a TFPL that is invariant under left-Wieland gyration is said to be \emph{stable}.

\begin{figure}[tbh]
\centering
\includegraphics[width=1\textwidth]{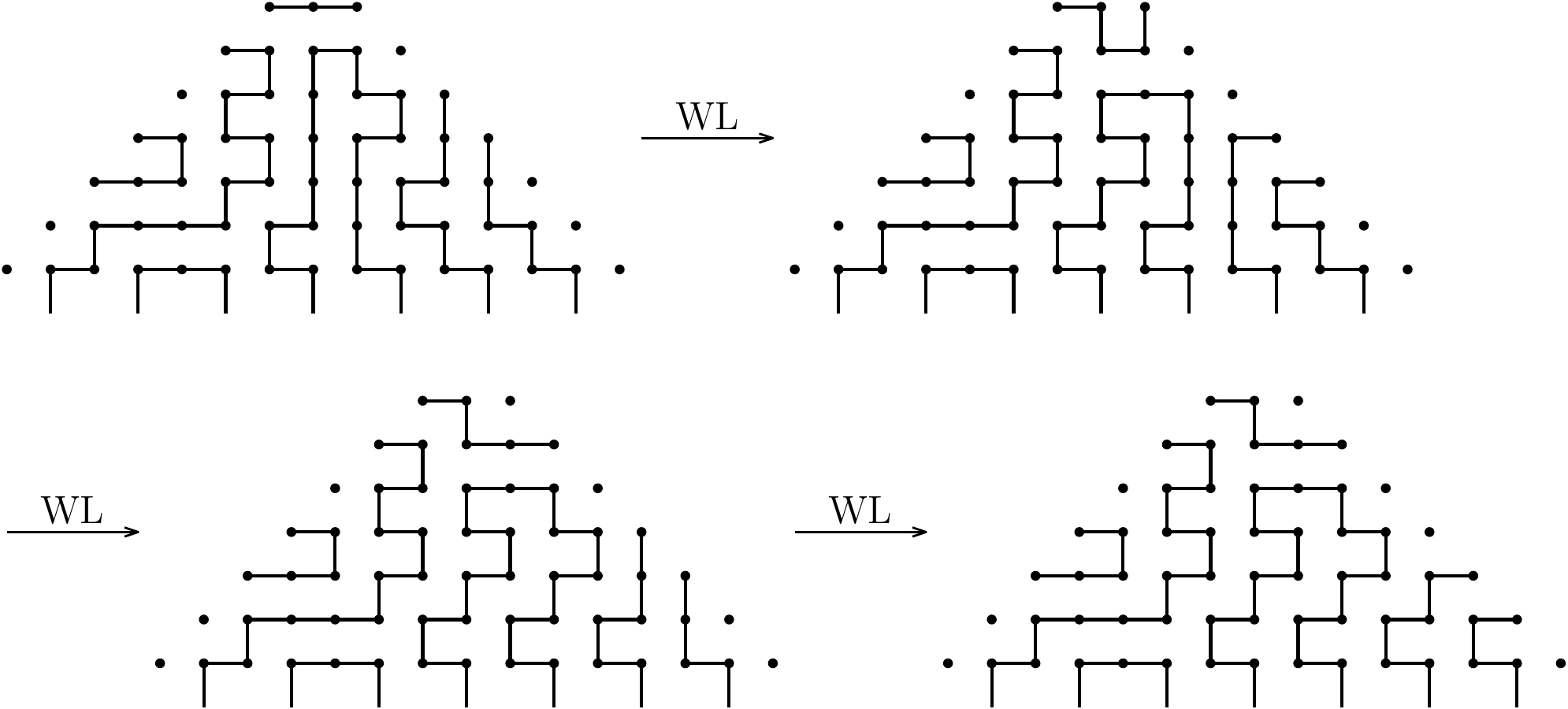}
\caption{A TFPL to which left-Wieland gyration is repeatedly applied.}
\label{Example:EventuallyStable}
\end{figure}

Given a TFPL $f$, the sequence $(\WL^m(f))_{m\geq 0}$ is eventually periodic since there are only finitely many TFPLs of a fixed size. The length of this period is in fact always 1.

\begin{Thm}\label{Thm:EventuallyStable}
Let $f$ be a TFPL of size $N$. Then $\WL^{2N-1}(f)$ is stable, so that the following holds for all $m\geq 2N-1$ :
\[
\WL^m(f)=\WL^{2N-1}(f).
\]
The same holds for right-Wieland gyration.
\end{Thm}

For that purpose, it is necessary to characterize TFPLs that are invariant under left-Wieland gyration. Note that a TFPL is invariant
under left-Wieland gyration if and only if it is invariant under right-Wieland gyration by Theorem~\ref{Thm:WielandBijectiveTFPL}.

%%%%%%%%%%%%%
\section{Stable TFPLs}
\label{Section:stable}
%%%%%%%%%%%%%

 From now on the vertices of $G^N$ are partitioned into odd and even
vertices in a chessboard manner such that by convention the vertices in $\mathcal{L}^N$ are odd. In our pictures, the odd vertices are depicted by circles and the even vertices by squares. An example of a TFPL where the partition of its vertices into odd and even vertices is indicated is depicted in \hbox{Figure~\ref{StableExcess1}}.

\begin{figure}[tbh]
\centering
\includegraphics[width=.5\textwidth]{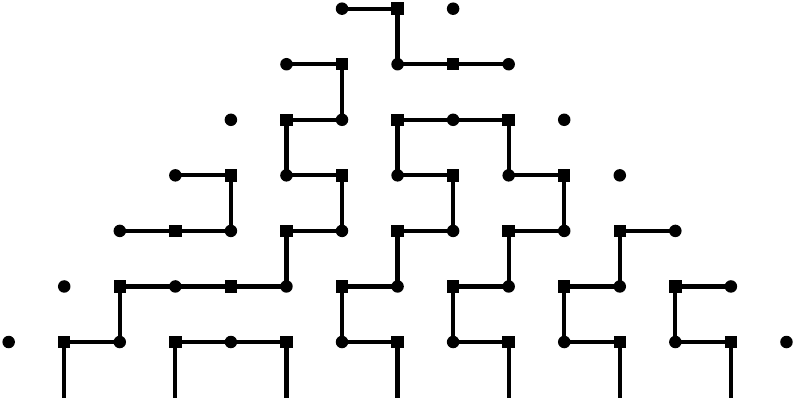}
\caption{The bottom right TFPL configuration of size 7 in Figure~\ref{Example:EventuallyStable} with its odd resp. even vertices illustrated by circles resp. squares.}
\label{StableExcess1}
\end{figure}

It will be proven that stable TFPLs can be characterized as follows:

\begin{Thm}\label{Thm:StableTFPL}
A TFPL is stable if and only if it contains no edge of the form $\vcenter{\hbox{\includegraphics[width=.0125\textwidth]{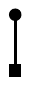}}}$.
\end{Thm}

The TFPL depicted in Figure~\ref{StableExcess1} is stable by Theorem~\ref{Thm:StableTFPL}.

\begin{Def}
An edge of the form $\vcenter{\hbox{\includegraphics[width=.0125\textwidth]{drifter.pdf}}}$ is called a \emph{drifter}. 
\end{Def}

In the following, the possible cells of a TFPL play an important role in the proofs. For convenience, notations for the 16 odd and 16 even cells of a TFPL are fixed. In Figure~\ref{ListCells} the chosen notation can be seen.

\begin{figure}[tbh]
\includegraphics[width=0.8\textwidth]{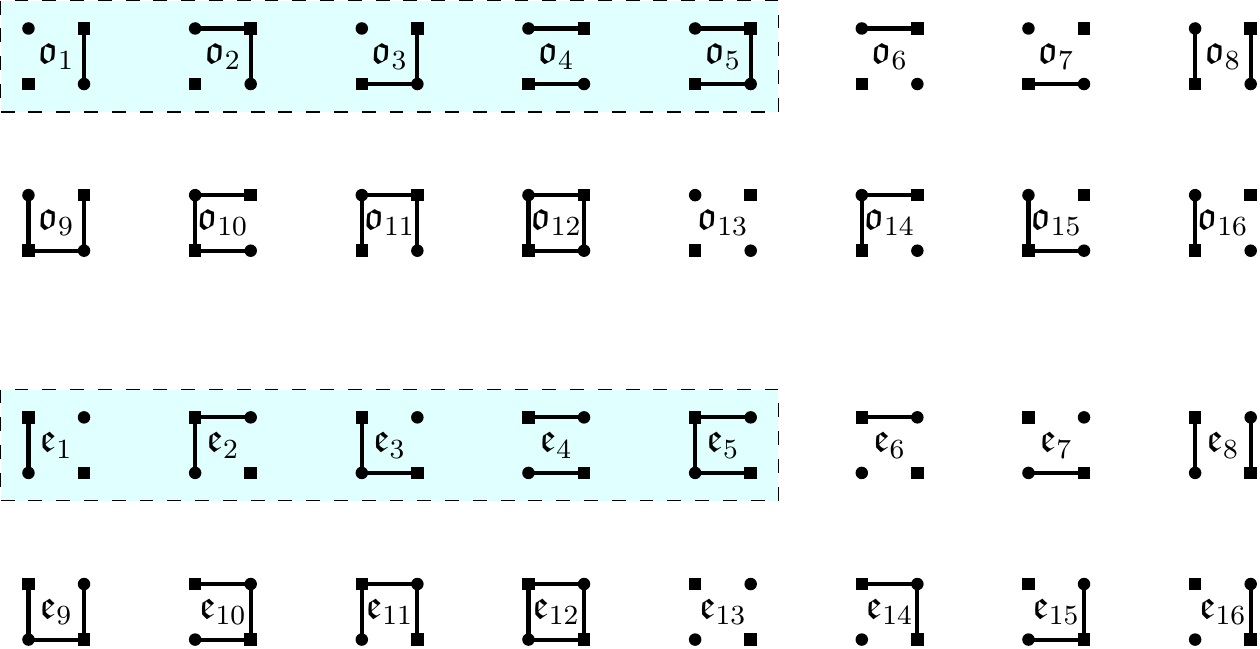}
\caption{The notations for the 16 odd and 16 even cells of a TFPL, with emphasis on the subsets $\mathfrak{O}=\{\mathfrak{o}_1,\mathfrak{o}_2,
\mathfrak{o}_3,\mathfrak{o}_4,\mathfrak{o}_5\}$ and $\mathfrak{E}=\{\mathfrak{e}_1,\mathfrak{e}_2, \mathfrak{e}_3,\mathfrak{e}_4,\mathfrak{e}_5\}$.}
\label{ListCells}
\end{figure}

%%%%%%%%%%%%%%%%%%%%
\subsection{Characterization of stable TFPLs}
\label{Sub:CharacterizationStable}

To prove Theorem~\ref{Thm:StableTFPL}, we will begin by showing that a TFPL containing a drifter is not stable.

\begin{Prop}\label{Prop:InstableTFPLs}
Let $f$ be a TFPL that contains a drifter. Then $\WL(f)\neq f$.
\end{Prop}
\begin{proof}
If $f$ contains a drifter incident to a vertex in $\mathcal{R}^N$, then by Lemma~\ref{WielandRightBoundaryTFPL} we know that the right boundaries of $f$ and $\WL(f)$ are different, so that necessarily $\WL(f)\neq f$.

We can now assume that no vertex in $\mathcal{R}^N$ is incident to a drifter. Let $\iota$ be a drifter in $f$ with maximal $x$-coordinate, and consider the odd cell $o$ in $f$ that contains $\iota$. Let $x$ be the top 
right vertex of $o$ and $y$ be the bottom right vertex of $o$. By the choice of $\iota$ the vertices $x$ and $y$ are not incident to a drifter. 
\begin{center}
\includegraphics[height=1cm]{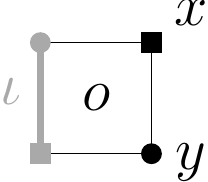}
\end{center}
Therefore, 
$o\in\{\mathfrak{o}_8, \mathfrak{o}_9, \mathfrak{o}_{10}, \mathfrak{o}_{11}, \mathfrak{o}_{12}\}$.
If $o$ is of the form $\mathfrak{o}_8$ or $\mathfrak{o}_{10}$ the vertex to the right of $x'$ is incident to a drifter in $\WL(f)$. In that case, $\WL(f)\neq f$ because the vertex to the right of $x$ in $f$ is not incident to a drifter by assumption.  
If $o$ is of the form $\mathfrak{o}_9$, $\mathfrak{o}_{11}$ or $\mathfrak{o}_{12}$, the vertices $x'$ and $y'$ are not adjacent in $\WL(f)$. Thus,
$\WL(f)\neq f$ because $x$ and $y$ are adjacent in $f$.
\end{proof}

To prove that a TFPL without a drifter is indeed stable, we need to determine the types of cells which may occur. Define $\mathfrak{O}=\{\mathfrak{o}_1,\mathfrak{o}_2,
\mathfrak{o}_3,\mathfrak{o}_4,\mathfrak{o}_5\}$ and $\mathfrak{E}=\{\mathfrak{e}_1,\mathfrak{e}_2,\mathfrak{e}_3,\mathfrak{e}_4,\mathfrak{e}_5\}$.

\begin{Lemma}\label{Lemma:CellsStableTFPL} If $f$ is a TFPL without drifters,   
then all interior odd cells belong to $\mathfrak{O}$ while all of its interior even cells belong to $\mathfrak{E}$.
\end{Lemma}

\begin{proof}
Let $f$ be a TFPL without a drifter, and $o$ be one of its interior odd cells. Since $o$ has no drifter, it can only belong to $\mathfrak{O}$ or have one of the types $\mathfrak{o}_6,\mathfrak{o}_7$ or 
$\mathfrak{o}_{13}$. But in types $\mathfrak{o}_6,\mathfrak{o}_{13}$ (\emph{resp.} $\mathfrak{o}_7,\mathfrak{o}_{13}$), there would exist an interior cell above $o$ ( \emph{resp.} below $o$) that contains a 
drifter, what is excluded.

The case of even cells is entirely analogous.
\end{proof}

Furthermore, in a TFPL with no drifter each odd cell has a uniquely determined even cell to its right. 

\begin{Lemma}\label{Lemma:PairsStableTFPL} Let $f$ be a TFPL without drifters, $o$ an odd cell of $f$ and $e$ the even cell of $f$ to the right of $o$. If $o$ and $e$ are interior,  then they can only occur as part of one of the following pairs:
\begin{center}
\includegraphics[width=0.8\textwidth]{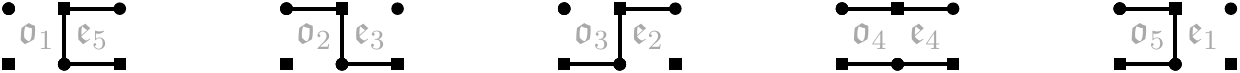}
\end{center} 
On the other hand, if $o$ or $e$ contains an external edge, then $o$ and $e$ can only occur as part of one of the following pairs:
\begin{center}
\includegraphics[width=.3\textwidth]{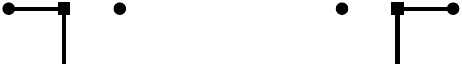} 
\end{center}
\end{Lemma}

\begin{proof}
Here, only the case when $o$ is an interior odd cell and $o=\mathfrak{o}_1$ is considered, the other cases being similar. Obviously, the cell $e$ cannot equal $\mathfrak{e}_4$. But it cannot equal $\mathfrak{e}_1$, $\mathfrak{e}_2$ or $\mathfrak{e}_3$ either, since otherwise one of the right vertices of $o$ would be be incident to a drifter. The only remaining possibility is that $e$ is of type $\mathfrak{e}_5$ by Lemma~\ref{Lemma:CellsStableTFPL}.
\end{proof}

We can now complete the proof of Theorem~\ref{Thm:StableTFPL} by showing that a TFPL without drifters is invariant under left-Wieland gyration.

\begin{Prop}\label{Prop007}
If $f$ is a TFPL without drifters, then $\WL(f)=f$. 
\end{Prop}
\begin{proof} Let $o$ be an odd cell of $f$ and $e$ be the even cell to its right. By \hbox{Lemma~\ref{Lemma:PairsStableTFPL}}, $e$ is uniquely determined by $o$. The crucial observation is that $e$ coincides with the image of $o$
under Wieland gyration. Thus, each even cell of $f$ and its corresponding even cell of $\WL(f)$ coincide.
By definition all edges and non-edges of $f$ incident to a vertex in $\mathcal{L}^N$ are preserved by left-Wieland gyration. In summary, $\WL(f)=f$.
\end{proof}

%%%%%%%%%%%%%%%%%%%%
\subsection{TFPLs are eventually stable under Wieland gyration}

In this section, we will prove Theorem~\ref{Thm:EventuallyStable}. The idea of the proof is the following: when applying left-Wieland gyration to a TFPL, the drifters of the TFPL are globally moved to the right. Thus, after a finite number of applications of left-Wieland gyration, all drifters eventually disappear through the right boundary. As a consequence of Theorem~\ref{Thm:StableTFPL} a stable TFPL is then obtained.

In a TFPL of size $N$, there are $2N+1$ columns of vertices which we label from left to right from $1$ to $2N+1$.

\begin{Prop}\label{Prop:EventuallyStable}
Let $f$ be a TFPL of size $N$ that contains a drifter in the $n$-th column but no drifter  in the columns $1,\ldots,n-1$ to its left. Then $\WL(f)$ contains no drifter in any of the columns $1,\ldots,n$.
\end{Prop}

%Note that in order to contain a drifter a TFPL has to be at least of size $2$.

\begin{proof} 
First of all, notice that by the definition of left-Wieland gyration, there is no vertex of $\mathcal{L}^{N\prime}$ incident to a drifter in $\WL(f)$.
By definition of $\WL$, the occurrence of a drifter in an even cell $e'$ of
$\WL(f)$ depends solely on the odd cell to the left of the corresponding even cell $e$ in $f$. By hypothesis, no odd cell of $f$ occurring to the left of the $(n-1)$-st column has a vertex incident to a drifter. It follows from the proof of Lemma ~\ref{Lemma:CellsStableTFPL} that all these odd cells belong to $\mathfrak{O}$. This entails that all even cells of $\WL(f)$ to the left of the $n$-th column
belong to $\mathfrak{E}$, and thus do not contain a drifter. Since these even cells cover all vertical edges in the columns $1,\ldots,n$, the proof is complete.
\end{proof}

\begin{proof}[Proof of Theorem~\ref{Thm:EventuallyStable}]
By immediate induction on the result of Proposition~\ref{Prop:EventuallyStable}, we know that the configuration  $\WL^{2N+1-n}(f)$ contains no drifter, and thus by Theorem~\ref{Thm:StableTFPL} it is stable under $\WL$, so that
\begin{equation}
\WL^m(f)=\WL^{2N+1-n}(f) 
\label{Equation:EventuallyStable}\end{equation}
for all $m\geq 2N+1-n$. Since the first column of vertices of a TFPL consists only of the vertex $L_1$, we have $n\geq 2$, which proves the theorem.
\end{proof}

%%%%%%%%%%%%%%%%%%%%%%%%
\section{Applications of Wieland gyration on TFPLs}
\label{Section:applications} 
%%%%%%%%%%%%%%%%%%%%%%%%

%%%%%%%%%%%
\subsection{Some linear relations}

The following was conjectured for Dyck words in~\cite{Thapper} and proved in~\cite{Nadeau1} using Wieland gyration on FPLs.

\begin{Prop}\label{Prop:LinearRelation}
Let $u$, $v$ and $w$ be binary words. Then
\begin{equation}
\sum_{u^+: u\stackrel{\mathrm{h}}{\longrightarrow}u^+} t_{u^+,v}^w=\sum_{v^+:v\stackrel{\mathrm{v}}{\longrightarrow}v^+}t_{u,v^+}^w.
\notag\end{equation}
\end{Prop}

\begin{proof}
 Indeed the function $\WL_u(\cdot)$ acts on all TFPLs with boundary $(u^+,v;w)$, while $\WR_v(\cdot)$ acts on TFPLs with boundary $(u,v^-;w)$. By Theorem~\ref{Thm:WielandBijectiveTFPL}, these functions are inverse of one another, and the result is obtained by taking cardinalities.
\end{proof}

%%%%%%%%%%%%%%%%%%%%%%%%%%%%%%%%%%%%%%%%%
\subsection{The inequality ~\eqref{Necessary3}}

This states that $\vert \lambda(u)\vert + \vert \lambda(v)\vert\leq\vert \lambda(w)\vert$ always holds for the boundaries $(u,v;w)$ of TFPLs. It was given in~\cite[Lemma 3.7]{Thapper} in the Dyck word case. Later, another proof in connection with TFPLs together with an orientation of the edges was given in~\cite{TFPL}. More precisely, it was shown there that in an oriented TFPL with boundary $(u,v;w)$, the quantity \hbox{$\vert\lambda(w)\vert-\vert\lambda(u)\vert-\vert\lambda(v)\vert$} counts 
occurrences of certain local patterns in the TFPL. 

We now give an independent proof based on the properties of Wieland gyration; the idea for this proof comes from the original one by Thapper, which can be seen as relying on  Wieland gyration on FPLs in an indirect way.

\begin{proof}[Proof of ~\eqref{Necessary3}] Let $f$ be a TFPL with boundary $(u,v;w)$. The proof is done by induction on $\vert\lambda(u)\vert$. In the case when $\vert\lambda(u)\vert=0$ we have \hbox{$\lambda(v)\subseteq\lambda(w)$} by Equation~\eqref{Necessary2}, which  implies \hbox{$\vert\lambda(v)\vert\leq\vert\lambda(w)\vert$}. 

Assume now $\vert\lambda(u)\vert\geq 1$. By removing a corner of $\lambda(u)$, there exists a Young diagram \hbox{$\lambda(u^-)\subseteq\lambda(u)$} with one cell less than $\lambda(u)$. In particular $\lambda(u)/\lambda(u^-)$ is a horizontal strip. 

We first want to prove that there exists $i>0$ such that $\WL_{u^-}^i(f)$ has right boundary $v^+\neq v$. Assume the contrary, that is the right boundary of 
$\WL_{u^-}^i(f)$ is $v$ for all $i>0$. Since there are only a finite number of TFPLs with boundary $(u^-,v;w)$, there exist integers $i_0,p>0$ such that
\[
\WL_{u^-}^{i_0+p}(f)=\WL_{u^-}^{i_0}(f).
\]
We can then apply $\WR^{i_0}_{v}$ to both sides of the identity, and by Theorem~\ref{Thm:WielandBijectiveTFPL} we obtain \hbox{$\WL_{u^-}^p(f)=f$}. 
But these configurations have left boundaries $u,u^-$ respectively and we assumed $u^-\neq u$, which is a contradiction.

Hence, let $i$ be a positive integer such that $\WL_{u^-}^i(f)$ has boundary $(u^-,v^+;w)$ where $v^+\neq v$. By
Proposition~\ref{imageWielandTFPL} we have $\lambda(v)\subsetneq\lambda(v^+)$ and therefore \hbox{$\vert\lambda(v)\vert+1\leq\vert\lambda(v^+)\vert$}.
Applying the induction hypothesis to $\WL_{u^-}^i(f)$ completes the proof:
\begin{equation*}
\vert\lambda(u)\vert+\vert\lambda(v)\vert=\vert\lambda(u^-)\vert+1+\vert\lambda(v)\vert\leq\vert\lambda(u^-)\vert+\vert\lambda(v^+)\vert\leq\vert\lambda(w)\vert.
\end{equation*}
\end{proof}

\subsection{Excesses $0,1$ and beyond}

For a TFPL with boundary  $(u,v;w)$, the nonnegative integer  \hbox{$\vert\lambda(w)\vert-\vert\lambda(u)\vert-\vert\lambda(v)\vert$} is called the \emph{excess} of $f$. 

\begin{Prop}
\label{prop:Excess0Stable}
If a TFPL has excess $0$, then it is stable.
\end{Prop}

\begin{proof}
It is a consequence of~\cite[Proposition 5.2]{TFPL} that TFPLs of excess $0$ do not contain drifters, so we can conclude with Theorem~\ref{Thm:StableTFPL}.
\end{proof}

These TFPLs are known to be counted by Littlewood--Richardson coefficients~\cite{Nadeau2,TFPL} as recalled in the introduction.

In~\cite{TFPL}, configurations of excess $1$ were also studied in some detail and enumerated. The authors defined a number of \emph{moves} on such (oriented) configurations in order to transform them, and ultimately reach a configuration of excess $0$. It turns out that these complicated moves are essentially equivalent to a simple application of $\WL$, at least when the configuration is not stable. 

 Therefore stable configurations should first be studied and enumerated, and other configurations may then be related to them through Wieland gyration, in order to find for instance linear relations between their cardinalities. The feasibility of such an approach is in particular supported by Theorem~\ref{Thm:EventuallyStable}.

\bibliography{WielandTFPL}
\bibliographystyle{plain}

\end{document}